\documentclass{article}

\usepackage{amsmath}
\usepackage{amssymb}
\usepackage{amsthm}
\usepackage{graphicx}
\usepackage{algorithm}
\usepackage{algpseudocode}
\usepackage{url}

\newtheorem{definition}{Definition}[section]
\newtheorem{theorem}[definition]{Theorem}
\newtheorem{lemma}[definition]{Lemma}
\newtheorem{corollary}[definition]{Corollary}

\newcommand{\lt}{\left}
\newcommand{\rt}{\right}
\newcommand\norm[1]{\left\lVert#1\right\rVert}
\newcommand{\abs}[1]{|#1|}
\newcommand{\ls}[1]{\lt(#1\rt)}

\newcommand{\bmat}[1]{\begin{bmatrix}#1\end{bmatrix}}

\DeclareMathOperator{\rank}{rank}
\DeclareMathOperator{\vol}{vol}

\newcommand{\RR}{\mathbb{R}}

\newcommand{\jset}{\mathcal{J}}
\newcommand{\eps}{\varepsilon}
\newcommand{\AI}{\mathbb{A}}
\newcommand{\B}{\mathcal{B}}
\newcommand{\N}{\mathcal{N}}
\newcommand{\e}{\+e}

\title{Rank Revealing Gaussian Elimination by the\\Maximum Volume Concept}
\date{\vspace{-5ex}}
\author{Lukas Schork\thanks{L.Schork@ed.ac.uk} \and Jacek
  Gondzio\thanks{J.Gondzio@ed.ac.uk}}

\begin{document}
\maketitle
\begin{center}
\textit{School of Mathematics, University of Edinburgh, Edinburgh EH9 3FD,
  Scotland, UK} \\[5ex]
\textbf{Technical Report ERGO 18-002, February 5, 2018}
\end{center}
\vspace{3ex}

\begin{abstract}
  A Gaussian elimination algorithm is presented that reveals the numerical rank
  of a matrix by yielding small entries in the Schur complement. The algorithm
  uses the maximum volume concept to find a square nonsingular submatrix of
  maximum dimension. The bounds on the revealed singular values are similar to
  the best known bounds for rank revealing $LU$ factorization, but in contrast
  to existing methods the algorithm does not make use of the normal matrix. An
  implementation for dense matrices is described whose computational cost is
  roughly twice the cost of an $LU$ factorization with complete pivoting.
  Because of its flexibility in choosing pivot elements, the algorithm is
  amenable to implementation with blocked memory access and for sparse matrices.
\end{abstract}

\section{Introduction}
\label{sec:introduction}

This paper is concerned with the problem to determine the rank of a matrix in
the numerical sense. Given $A\in\RR^{m\times n}$ and a tolerance $\eps$, the
task is to determine an index $r$ such that $\sigma_r\ge\eps$ and
$\sigma_{r+1}=\mathcal{O}(\eps)$, where
$\sigma_1\ge\ldots\ge\sigma_d\ge\sigma_{d+1}:=0$ ($d=\min(m,n)$) are the
singular values of $A$. Our definition of numerical rank relaxes the condition
$\sigma_r\ge\eps>\sigma_{r+1}$ since the latter can only be achieved by
computing the singular values. Additionally to the rank $r$, we want to identify
an $r\times r$ submatrix of $A$ whose minimum singular value is not too much
smaller than $\sigma_r$.

It is well known that Gaussian elimination with complete pivoting may not detect
a near singularity. For the example from \cite{peters1970},
\begin{equation}
  \label{eq:peters}
  A = \bmat{1 & -1 & \cdots & -1 & -1\\
    & 1 & & & -1 \\ & & \ddots & & \vdots \\ & & & 1 & -1 \\ & & & & 1}
  \in \RR^{m\times m},
\end{equation}
complete pivoting allows to choose the diagonal elements as pivots, so that no
eliminations are needed and $A$ is determined to be of full rank. It is not
revealed that $\sigma_m(A)=\mathcal{O}(2^{-m})$ (see \cite[Section~5]{pan2000})
and the numerical rank of $A$ to be $m-1$ for $m$ moderately large.

The algorithm presented in this paper is based on Gaussian elimination and is
rank revealing in the above definition. It finds a nonsingular submatrix
$A_{11}$ of $A$ such that
\begin{equation*}
  \norm{A/A_{11}}_C\le\beta \quad \text{and} \quad
  \norm{A_{11}^{-1}}_C\le\beta^{-1}
\end{equation*}
for a given parameter $\beta>0$. Here $\norm{\cdot}_C$ is the maximum absolute
entry of a matrix and $A/A_{11}$ is the Schur complement of $A_{11}$ in $A$. It
will be shown that for $\beta=\max(m,n)\eps$ the dimension of $A_{11}$ reveals
the numerical rank of $A$. A lower bound on the minimum singular value of
$A_{11}$ will be derived in terms of $\sigma_r(A)$. Applied to the matrix
\eqref{eq:peters}, the algorithm selects the upper right $(m-1)\times(m-1)$
block as $A_{11}$ for which $\norm{A/A_{11}}_C=\mathcal{O}(2^{-m})$ and
$\sigma_{\min}(A_{11})\approx\sigma_{m-1}(A)$.

To find $A_{11}$, the algorithm selects an $m\times m$ basis matrix $\AI_\B$ of
local maximum volume in $\AI=\bmat{A& \beta I_m}$, where $I_m$ is the identity
matrix of dimension $m$. The concept of maximum volume has been used before in
rank revealing factorizations and related topics, see
\cite{pan2000,goreinov2001,goreinov2010} and the references therein. The novelty
of our algorithm is to work on the matrix $\AI$ rather than $A$ itself. $A_{11}$
will be defined by means of the columns of $A$ and the columns of $\beta I_m$
which compose $\AI_\B$. It will be shown that $\norm{A/A_{11}}_2$ and
$\sigma_{min}(A_{11})$ satisfy bounds in terms of the singular values of $A$
that are very similar to the best known bounds for the rank revealing $LU$
factorization \cite{pan2000}.

A rank revealing factorization based on the maximum volume concept that yields a
square nonsingular submatrix has also been derived by Pan \cite{pan2000}. Pan's
method first chooses a column subset of $A$ by utilizing the normal matrix
$A^TA$, and then chooses a square submatrix within these columns. A detailed
comparison to our method is given. While the resulting submatrices have the same
rank revealing properties, an advantage of our method is not to use the normal
matrix but instead to use pivot operations on $\AI$ only. This is particularly
relevant with regard to an implementation for sparse matrices.

An implementation of the proposed algorithm for dense matrices is described. It
requires roughly twice the computational cost than an $LU$ factorization of
$\AI$ with complete pivoting. Comparisons to the singular value decomposition on
a set of rank deficient matrices show that the rank detection is reliable and
that the condition number of the selected submatrices is close to
$\sigma_1(A)/\sigma_r(A)$.

Throughout the paper $A$ is an $m\times n$ matrix and $A_{11}$ is a square
nonsingular submatrix. It is assumed that $A$ has been permuted so that
\begin{equation}
  \label{eq:partitionedA}
  A = \bmat{A_{11} & A_{12} \\ A_{21} & A_{22}}.
\end{equation}
The Schur complement of $A_{11}$ in $A$ is
\begin{equation*}
  A/A_{11} = A_{22} - A_{21}A_{11}^{-1}A_{12}.
\end{equation*}
$\sigma_k(\cdot)$ denotes the $k$-th singular value of a matrix, where the
singular values are ordered nonincreasingly. $\norm{\cdot}_2$ and
$\norm{\cdot}_C$ are the maximum singular value and the maximum absolute entry
norm of a matrix. They satisfy the relation
\begin{equation*}
  \norm{A}_C \le \norm{A}_2 \le \sqrt{mn}\norm{A}_C.
\end{equation*}
For an index set $\jset$, $A_\jset$ is the matrix composed of the columns of $A$
indexed by $\jset$. A basis $\B$ for $\AI\in\RR^{m\times(n+m)}$ is an index set
such that the basis matrix $\AI_\B$ is square and nonsingular (requiring that
$\AI$ has rank $m$). Associated with $\B$ is the nonbasic set
$\N=\{1,\ldots,n+m\}\setminus\B$. Vectors are notated in bold lower case, where
$\e_j$ is the $j$-th unit vector. Expression like $\abs{A}$ and $\abs{\+b}$ are
meant componentwise.

\section{Maximum Volume Concept}
\label{sec:maxvolume}

The volume of a matrix of arbitrary dimension and rank is introduced in
\cite{ben-israel1992}. This paper uses the definition from \cite{pan2000}, which
differs in that the volume of a rank deficient matrix is zero.

\begin{definition}
  \label{def:volume}
  For $A\in\RR^{m\times n}$ with singular values
  $\sigma_1\ge\ldots\ge\sigma_d\ge0$ ($d=\min(m,n)$), the volume of $A$ is
  defined by
  \begin{equation*}
    \vol(A) = \sigma_1\cdots\sigma_d.
  \end{equation*}
\end{definition}

In particular, the volume of a square matrix is the absolute value of its
determinant.

\begin{definition}
  \label{def:maxvolume}
  Let $A\in\RR^{m\times n}$ and $\rho\ge1$.
  \begin{enumerate}
  \item Let $B$ be a $k\times k$ submatrix of $A$. $\vol(B)(\neq0)$ is said to
    be a global $\rho$-maximum volume in $A$ if
    \begin{equation}
      \label{eq:maxvolume}
      \rho\vol(B) \ge \vol(B')
    \end{equation}
    for all $k\times k$ submatrices $B'$ of $A$.
  \item Let $B$ be formed by k columns (rows) of $A$. $\vol(B)(\neq0)$ is said
    to be a local $\rho$-maximum volume in $A$ if \eqref{eq:maxvolume}
    holds for any $B'$ that is obtained by replacing one column (row) of $B$ by
    a column (row) of $A$ which is not in $B$.
  \item Let $B$ be a $k\times k$ submatrix ($k<\min(m,n)$) of $A$.
    $\vol(B)(\neq0)$ is said to be a local $\rho$-maximum volume in $A$ if it is
    a global $\rho$-maximum volume in all $(k+1)\times(k+1)$ submatrices of $A$
    which contain $B$.
  \end{enumerate}
\end{definition}

The important concept in the theory of rank revealing factorizations is the
local maximum volume. The definition \ref{def:maxvolume}(ii) is from
\cite{pan2000} and \ref{def:maxvolume}(iii) is the natural extension to square
submatrices of any dimension. It is equivalent to saying that $A_{11}$ has local
$\rho$-maximum volume in \eqref{eq:partitionedA} if the volume of the $(1,1)$
block cannot be increased by more than a factor $\rho$ by interchanging two
columns and/or two rows.

Finding a submatrix of local maximum volume will make use of column and row
exchanges. The following lemmas provide fomulas for the change of volume when a
column and/or row is replaced in a square nonsingular matrix.

\begin{lemma}
  \label{lem:colchange}
  Let $A_{11}$ be $k\times k$ nonsingular and $A_{11}'$ be obtained by replacing
  column $j$ by the vector $\+b$. Then
  \begin{equation*}
    \frac{\vol(A_{11}')}{\vol(A_{11})} = \abs{A_{11}^{-1}\+b}_j.
  \end{equation*}
  In particular, $A_{11}$ in \eqref{eq:partitionedA} has local $\rho$-maximum
  volume in its block row and block column if and only if
  $\norm{A_{11}^{-1}A_{12}}_C\le\rho$ and $\norm{A_{21}A_{11}^{-1}}_C\le\rho$,
  respectively.
\end{lemma}
\begin{proof}
  \begin{align*}
    A_{11}' &= A_{11} - A_{11}\e_j\e_j^T + \+b\e_j^T \\
    &= A_{11} (I_k - \e_j\e_j^T + A_{11}^{-1}\+b\e_j^T).
  \end{align*}
  The expression in parenthesis is the identity matrix with column $j$ replaced
  by $A_{11}^{-1}\+b$. Therefore
  \begin{align*}
    \det(A_{11}') &= \det(A_{11}) \det(I_k - \e_j\e_j^T + A_{11}^{-1}\+b\e_j^T)
    \\
    &= \det(A_{11}) (A_{11}^{-1}\+b)_j
  \end{align*}
  and taking absolute values completes the proof.
\end{proof}

\begin{lemma}
  \label{lem:rowcolremove}
  Let $\hat{A}$ be square and nonsingular and $B$ be obtained by removing row
  $i$ and column $j$. Then
  \begin{equation*}
    \frac{\vol(B)}{\vol(\hat{A})} = \abs{\hat{A}^{-1}}_{j,i}.
  \end{equation*}
  In particular, $B$ has $\rho$-maximum volume in $\hat{A}$ if and only if
  $\rho\abs{\hat{A}^{-1}}_{j,i}\ge\norm{\hat{A}^{-1}}_C$.
\end{lemma}
\begin{proof}
  By Cramer's rule
  \begin{equation*}
    (\hat{A}^{-1})_{j,i} = (\hat{A}^{-1}\e_i)_j
    = \frac{\det(\hat{A}-\hat{A}\e_j\e_j^T+\e_i\e_j^T)}{\det(\hat{A})}.
  \end{equation*}
  Since the matrix whose determinant is taken in the numerator has unit column
  $\e_i$ in position $j$, by Laplace's formula
  \begin{equation*}
    \det(\hat{A}-\hat{A}\e_j\e_j^T+\e_i\e_j^T) = (-1)^{i+j} \det(B).
  \end{equation*}
  Substituting into the previous expression and taking absolute values completes
  the proof.
\end{proof}

\begin{lemma}
  \label{lem:rowcolchange}
  Let $A_{11}$ be $k\times k$ nonsingular and
  \begin{equation}
    \label{eq:rowcolchange:Ahat}
    \hat{A} = \bmat{A_{11} & \+b \\ \+c^T & \alpha}.
  \end{equation}
  Let $\gamma=\hat{A}/A_{11}$ and $A_{11}''$ be the leading $k\times k$ block of
  $\hat{A}$ after interchanging columns $k+1$ and $j$ ($1\le j\le k$) and rows
  $k+1$ and $i$ ($1\le i\le k$). Then
  \begin{equation}
    \label{eq:rowcolchange:ratio}
    \frac{\vol(A_{11}'')}{\vol(A_{11})} =
    \abs{\gamma(A_{11}^{-1})_{j,i} + (A_{11}^{-1}\+b)_j(A_{11}^{-T}\+c)_i}.
  \end{equation}
\end{lemma}
\begin{proof}
  Firstly consider that $\hat{A}$ is singular, in which case $\rank(\hat{A})=k$
  and $\gamma=0$. If $\abs{A_{11}^{-1}\+b}_j=0$, then the first $k$ columns of
  $\hat{A}$ after the interchanges have rank $k-1$. Hence $A_{11}''$ must be
  singular and both sides of \eqref{eq:rowcolchange:ratio} are zero. Otherwise
  let $A_{11}'$ be obtained from $A_{11}$ by replacing column $j$ by the vector
  $\+b$. Then, by Lemma~\ref{lem:colchange},
  \begin{equation*}
    \vol(A_{11}') = \vol(A_{11}) \abs{A_{11}^{-1}\+b}_j.
  \end{equation*}
  Let $\+c'$ be obtained from $\+c$ be replacing the $j$-th entry by $\alpha$.
  Because $\hat{A}$ is singular,
  \begin{equation*}
    (A_{11}')^{-T}\+c' = A_{11}^{-T}\+c.
  \end{equation*}
  Therefore, by Lemma~\ref{lem:colchange},
  \begin{align*}
    \vol(A_{11}'') &= \vol(A_{11}') \abs{(A_{11}')^{-T}\+c'}_i \\
    &= \vol(A_{11}) \abs{A_{11}^{-1}\+b}_j \abs{A_{11}^{-T}\+c}_i.
  \end{align*}

  Secondly consider that $\hat{A}$ is nonsingular, in which case $\gamma\neq0$.
  Then
  \begin{align}
    \label{eq:Ahatinv}
    \hat{A}^{-1} &= \bmat{H& \+f\\ \+g^T& \gamma^{-1}}, \\
    \intertext{where}
    \+f &= -\gamma^{-1}A_{11}^{-1}\+b, \notag \\
    \+g &= -\gamma^{-1}A_{11}^{-T}\+c, \notag \\
    H &= A_{11}^{-1} + \gamma\+f\+g^T. \notag
  \end{align}
  It follows from Lemma~\ref{lem:rowcolremove} that
  \begin{align*}
    \vol(A_{11}) &= \abs{\gamma^{-1}} \vol(\hat{A}), \\
    \vol(A_{11}'') &= \abs{H_{j,i}} \vol(\hat{A}).
  \end{align*}
  Therefore
  \begin{align*}
    \frac{\vol(A_{11}'')}{\vol(A_{11})}
    &= \frac{\abs{(A_{11}^{-1})_{j,i} +
        \gamma^{-1}(A_{11}^{-1}\+b)_j(A_{11}^{-T}\+c)_i}} {\abs{\gamma^{-1}}} \\
    &= \abs{\gamma(A_{11}^{-1})_{j,i} + (A_{11}^{-1}\+b)_j(A_{11}^{-T}\+c)_i}.
  \end{align*}
\end{proof}

\section{Rank Revealing Algorithm}
\label{sec:algorithm}

This section presents the algorithm for selecting the submatrix $A_{11}$ whose
dimension reveals the numerical rank of $A$. Instead of selecting the row and
column subsets directly, the algorithm selects a basis matrix of
$\AI=\bmat{A&\beta I_m}$. The columns of $A$ and $\beta I_m$ in $\AI$ are termed
\emph{structural} and \emph{logical}, respectively. Assume that $\B$, $\N$ is a
basic-nonbasic partitioning of the columns of $\AI$ and
\begin{equation}
  \label{eq:BN}
  \AI_\B = \bmat{A_{11}&0\\ A_{21}&\beta I_{m-k}}, \quad
  \AI_\N = \bmat{A_{12}&\beta I_k\\ A_{22}&0},
\end{equation}
where the rightmost $m-k$ and $k$ columns of $\AI_\B$ and $\AI_\N$ are logical
(the indices in $\B$ and $\N$ can always be permuted to obtain that form). The
partitioning uniquely determines $A_{11}$. Therefore any basis for $\AI$
determines a square nonsingular $A_{11}$.

To obtain $A_{11}$ with the desired properties, it will turn out that $\AI_\B$
must have local $\rho$-maximum volume in $\AI$. An algorithm for finding a basis
matrix of local maximum volume is given in \cite{goreinov2010}.
Algorithm~\ref{alg:maxvolume} is a generic version that leaves some flexibility
to the implementation by not specifying how to choose $(p,q)$ in
line~\ref{line:choosepq} in case there is more than one candidate. In
particular, it is not necessary to scan the entire matrix $\AI_\B^{-1}\AI_\N$ in
every iteration or even to compute it explicitly.

\begin{algorithm}
  \caption{\texttt{find\_submatrix}}
  \label{alg:maxvolume}
  \begin{algorithmic}[1]
    \Require $A\in\RR^{m\times n}$, $\rho\ge1$, $\beta>0$
    \State Build $\AI=\bmat{A&\beta I_m}$
    \State Initialize $\B=\{n+1,\ldots,n+m\}$, $\N=\{1,\ldots,n\}$
    \Loop
    \State Let $M=\AI_\B^{-1}\AI_\N$
    \If{$\norm{M}_C\le\rho$}
    \State Stop
    \EndIf
    \State Choose $(p,q)$ such that $\abs{M}_{p,q}>\rho$
    \label{line:choosepq}
    \State $\B_p\gets\N_q$
    \EndLoop
    \State Build $A_{11}$ from \eqref{eq:BN}
    \State Let $r$ be the dimension of $A_{11}$
  \end{algorithmic}
\end{algorithm}

\begin{lemma}
  \label{lem:algorithm}
  Algorithm~\ref{alg:maxvolume} terminates in a finite number of iterations.
  The resulting $A_{11}$ has local $(2\rho^2)$-maximum volume in $A$ and
  \begin{equation}
    \label{eq:betabound}
    \norm{A/A_{11}}_C\le\rho\beta, \quad
    \norm{A_{11}^{-1}}_C\le\rho\beta^{-1}.
  \end{equation}
\end{lemma}
\begin{proof}
  Each basis update in Algorithm~\ref{alg:maxvolume} increases the volume of
  $\AI_\B$ by a factor greater than $1$. Therefore a basis cannot repeat and the
  algorithm terminates in a finite number of iterations. When the algorithm
  terminates, all entries of
  \begin{equation}
    \label{eq:tbl}
    \AI_\B^{-1}\AI_\N = \bmat{A_{11}^{-1}A_{12}& \beta
      A_{11}^{-1}\\ \beta^{-1}A/A_{11}& -A_{21}A_{11}^{-1}}
  \end{equation}
  are bounded by $\rho$ in absolute value. This means that $A_{11}$ has local
  $\rho$-maximum volume in its block row and block column, and
  $\norm{A/A_{11}}_C\le\rho\beta$ and $\norm{A_{11}^{-1}}_C\le\rho\beta^{-1}$.
  If $r=m$ (i.\,e.\ $\B$ contains only structural columns), then $A_{11}$ has
  local $\rho$-maximum volume in $A$. Otherwise consider any submatrix of $A$ of
  the form \eqref{eq:rowcolchange:Ahat}. The right-hand side in
  \eqref{eq:rowcolchange:ratio} is bounded by
  \begin{equation*}
    \abs{\gamma(A_{11}^{-1})_{j,i}+(A_{11}^{-1}\+b)_j(A_{11}^{-T}\+c)_i}
    \le \rho\beta\rho\beta^{-1} + \rho\rho = 2\rho^2.
  \end{equation*}
  Therefore $A_{11}$ has local $(2\rho^2)$-maximum volume in $A$.
\end{proof}

The numerical rank of $A$ is determined by Algorithm~\ref{alg:maxvolume} as the
dimension of $A_{11}$. It follows from the interlacing property of the singular
values \cite[Corollary~8.6.3]{golub1996} that for any $k\times k$ submatrix $B$
of $A$,
\begin{equation*}
  \norm{B^{-1}}_C \ge \frac{1}{k}\norm{B^{-1}}_2 = \frac{1}{k\sigma_{\min}(B)}
  \ge \frac{1}{k\sigma_k(A)}.
\end{equation*}
If we choose $\beta\ge\max(m,n)\eps\rho$ in Algorithm~\ref{alg:maxvolume},
then it is guaranteed that
\begin{equation*}
  \frac{1}{r\sigma_r(A)} \le \norm{A_{11}^{-1}}_C \le \frac{1}{\max(m,n)\eps}
\end{equation*}
and therefore $\sigma_r(A)\ge\eps$ as desired. (Using $\min(m,n)$ instead of
$\max(m,n)$ would be sufficient.) On the other hand, from
\cite[Theorem~2.7]{pan2000},
\begin{equation*}
  \norm{A/B}_2 \ge \sigma_{k+1}(A)
\end{equation*}
for any $k\times k$ submatrix $B$ of $A$. Therefore
\begin{equation*}
  \sigma_{k+1}(A) \le \norm{A/A_{11}}_C\sqrt{(m-r)(n-r)} \le \beta \rho
  \sqrt{(m-r)(n-r)}.
\end{equation*}
In contrast to the singular value decomposition, Algorithm~\ref{alg:maxvolume}
cannot determine $r$ such that $\sigma_r(A)\ge\eps>\sigma_{r+1}(A)$. It can only
guarantee the first inequality and a bound on $\sigma_{r+1}(A)$ in terms of
$\eps$ and the dimension of $A$. In our definition this is sufficient for a rank
revealing factorization. In practice, a reasonable choice for $\beta$ might be
\begin{equation}
  \label{eq:tol}
  \beta = \max(m,n) \eps_{\text{mach}} \norm{A}_C,
\end{equation}
where $\eps_{\text{mach}}$ is the relative machine precision.

\section{Bounds on $\sigma_{\min}(A_{11})$ and $\norm{A/A_{11}}_2$}
\label{sec:rankrev}

The discussion so far has shown that $A_{11}$ that satisfies
\eqref{eq:betabound} reveals the numerical rank of $A$. It remains to be shown
that the minimum singular value of $A_{11}$ is close to $\sigma_r(A)$ for
$A_{11}$ obtained from Algorithm~\ref{alg:maxvolume}. This section derives
bounds on $\sigma_{min}(A_{11})$ and $\norm{A/A_{11}}_2$ in terms of the
singular values of $A$ that hold for any local maximum volume submatrix. More
specifically, the following theorem is proved.

\begin{theorem}
  \label{thm:rankrev}
  Let $A_{11}$ be $k\times k$ nonsingular and have local $(2\rho^2)$-maximum
  volume in $A$. Then
  \begin{align}
    \label{eq:maxvolume:lowerbound}
    &\sigma_k(A) \ge \sigma_{\min}(A_{11}) \ge 
    \frac{1}{2\rho^2 k\sqrt{(m-k+1)(n-k+1)}} \sigma_k(A), \\
    \label{eq:maxvolume:upperbound}
    &\sigma_{k+1}(A) \le \norm{A/A_{11}}_2 \le
    2\rho^2(k+1)\sqrt{(m-k)(n-k)}\sigma_{k+1}(A).
  \end{align}
\end{theorem}

The first inequalities in \eqref{eq:maxvolume:lowerbound} and
\eqref{eq:maxvolume:upperbound} hold true for any $k\times k$ submatrix of $A$,
whereas the second inequalities require the maximum volume property.
\eqref{eq:maxvolume:upperbound} is proved in \cite{goreinov2001} under the
assumption that $A_{11}$ has global $\rho$-maximum volume in $A$. Interestingly,
the proof given there goes through unchanged if $A_{11}$ has local
$\rho$-maximum volume as defined in this paper. The proof is given for
completeness. The proof for \eqref{eq:maxvolume:lowerbound} is new to the
authors.

\begin{lemma}
  \label{lem:upperbound}
  Let $A\in\RR^{m\times n}$ and $A_{11}$ be a nonsingular $k\times k$ submatrix
  ($k<\min(m,n)$) of local $\rho$-maximum volume. Then
  \begin{equation*}
    \norm{A/A_{11}}_C \le \rho(k+1)\sigma_{k+1}(A).
  \end{equation*}
\end{lemma}
\begin{proof}[Proof (from \cite{goreinov2001}).]
  Consider any $(k+1)\times(k+1)$ submatrix of $A$ of the form
  \begin{equation*}
    \hat{A} = \bmat{A_{11}& \+b\\ \+c^T& \alpha}.
  \end{equation*}
  Then $\gamma=\alpha-\+c^TA_{11}^{-1}\+b$ is an entry of $A/A_{11}$ and each
  entry of $A/A_{11}$ has this form for a particular $\hat{A}$. Therefore it
  suffices to show that $\abs{\gamma}\le\rho(k+1)\sigma_{k+1}(A)$.

  If $\hat{A}$ is singular, then $\gamma=0$ and the claim is trivial. Otherwise,
  because $A_{11}$ has $\rho$-maximum volume in $\hat{A}$, by
  Lemma~\ref{lem:rowcolremove} and \eqref{eq:Ahatinv},
  \begin{equation*}
    \rho \abs{\gamma^{-1}} \ge \norm{\hat{A}^{-1}}_C.
  \end{equation*}
  It follows that
  \begin{equation*}
    \abs{\gamma}
    \le \rho \frac{1}{\norm{\hat{A}^{-1}}_C}
    \le \rho \frac{k+1}{\norm{\hat{A}^{-1}}_2}
    = \rho(k+1)\sigma_{k+1}(\hat{A})
    \le \rho(k+1)\sigma_{k+1}(A),
  \end{equation*}
  where the last inequality comes from the interlacing property of singular
  values \cite[Corollary~8.6.3]{golub1996}.
\end{proof}

\begin{corollary}
  \label{cor:upperbound}
  Let $A\in\RR^{m\times n}$ and $A_{11}$ be a nonsingular $k\times k$ submatrix
  ($k<\min(m,n)$) of local $\rho$-maximum volume. Then
  \begin{equation*}
    \sigma_{k+1}(A) \le \norm{A/A_{11}}_2 \le
    \rho(k+1)\sqrt{(m-k)(n-k)}\sigma_{k+1}(A).
  \end{equation*}
\end{corollary}
\begin{proof}
  The first inequality is proved in \cite[Theorem~2.7]{pan2000}. The
  second inequality follows from Lemma~\ref{lem:upperbound}.
\end{proof}

\begin{lemma}
  \label{lem:lowerbound}
  Let $A\in\RR^{m\times n}$ and $A_{11}$ be a nonsingular $k\times k$ submatrix
  of local $\rho$-maximum volume. Then
  \begin{equation*}
    \sigma_k(A) \le \rho k \sqrt{(m-k+1)(n-k+1)} \sigma_k(A_{11}).
  \end{equation*}
\end{lemma}
\begin{proof}
  If $k=1$, then $A_{11}$ is scalar and because of local $\rho$-maximum volume
  it satisfies $\rho\abs{A_{11}}\ge\norm{A}_C$. Therefore
  \begin{equation*}
    \sigma_1(A) \le \sqrt{mn}\norm{A}_C \le \sqrt{mn}\rho\abs{A_{11}} = \rho
    \sqrt{mn}\sigma_1(A_{11}).
  \end{equation*}
  If $k>1$, let $B$ be a $(k-1)\times(k-1)$ submatrix of $A_{11}$ with maximum
  volume in $A_{11}$. In particular $B$ is nonsingular. Consider any $k\times k$
  submatrix of $A$ of the form
  \begin{equation*}
    A_{11}'' = \bmat{B& \+b\\ \+c^T& \alpha}.
  \end{equation*}
  Because $A_{11}''$ differs from $A_{11}$ by at most one row and one column,
  and because $A_{11}$ has local $\rho$-maximum volume in $A$,
  \begin{equation*}
    \rho\vol(A_{11}) \ge \vol(A_{11}'').
  \end{equation*}
  From the determinant property of the Schur complement,
  \begin{equation*}
    \det(A_{11}) = \det(B) \det(A_{11}/B),
  \end{equation*}
  it follows that
  \begin{equation*}
    \rho \abs{A_{11}/B} = \rho \frac{\vol(A_{11})}{\vol(B)} \ge
    \frac{\vol(A_{11}'')}{\vol(B)} = \abs{A_{11}''/B}.
  \end{equation*}
  Since $A_{11}''/B$ is an entry of $A/B$ and each entry of $A/B$ has this form
  for a particular $A_{11}''$, it follows that
  \begin{equation*}
    \rho \abs{A_{11}/B} \ge \norm{A/B}_C.
  \end{equation*}
  Therefore
  \begin{align*}
    \sigma_k(A) \le \norm{A/B}_2 &\le \sqrt{(m-k+1)(n-k+1)} \norm{A/B}_C \\
    &\le \rho \sqrt{(m-k+1)(n-k+1)} \abs{A_{11}/B} \\ 
    &\le \rho \sqrt{(m-k+1)(n-k+1)} k \sigma_k(A_{11}),
  \end{align*}
  where the first inequality is from \cite[Theorem~2.7]{pan2000} and the last
  inequality from Lemma~\ref{lem:upperbound} and the fact that $B$ has maximum
  volume in $A_{11}$.
\end{proof}

\begin{corollary}
  \label{cor:lowerbound}
  Let $A\in\RR^{m\times n}$ and $A_{11}$ be a nonsingular $k\times k$ submatrix
  of local $\rho$-maximum volume. Then
  \begin{equation*}
    \sigma_k(A) \ge \sigma_{\min}(A_{11}) \ge
    \frac{1}{\rho k\sqrt{(m-k+1)(n-k+1)}} \sigma_k(A).
  \end{equation*}
\end{corollary}
\begin{proof}
  The first inequality comes from the interlacing property of singular
  values \cite[Corollary~8.6.3]{golub1996}.
  The second inequality follows from Lemma~\ref{lem:lowerbound}. 
\end{proof}

Theorem~\ref{thm:rankrev} follows from Corollaries
\ref{cor:lowerbound} and \ref{cor:upperbound}.

\section{Comparison to Pan's Method}
\label{sec:pan}

Pan \cite{pan2000} uses the maximum volume concept in a rank revealing
factorization algorithm based on Gaussian elimination, which yields a submatrix
$A_{11}$ that has very similar properties to the submatrix obtained from
Algorithm~\ref{alg:maxvolume}. This section compares the two methods regarding
their use of the maximum volume property and possible implementations.

Given $A\in\RR^{m\times n}$, $\rho\ge1$ and $k\le\rank(A)$, Pan's method first
chooses an $m\times k$ submatrix $A_\jset$ of local $\rho$-maximum volume in
$A$, and then a $k\times k$ submatrix $A_{11}$ of local $\rho$-maximum volume in
$A_\jset$. We say that $A_{11}$ has \emph{normal $\rho$-maximum volume} in $A$
to distinguish it from our definition of local maximum volume. Theorem~3.8 in
\cite{pan2000} proves the following bounds on the singular values for $m=n$,
which are almost identical to those in Theorem~\ref{thm:rankrev}:
\begin{align*}
  &\sigma_k(A) \ge \sigma_{\min}(A_{11}) \ge \frac{1}{k(n-k)\rho^2+1}
  \sigma_k(A), \\
  &\sigma_{k+1}(A) \le \norm{A/A_{11}}_2 \le\lt(k(n-k)\rho^2+1\rt)
  \sigma_{k+1}(A).
\end{align*}

Pan's method uses the normal matrix $A^TA$ to find a column subset of local
$\rho$-maximum volume in $A$. The algorithm applies symmetric row and column
interchanges to $A^TA$ until the volume of the leading $k\times k$ block cannot
be increased by more than a factor $\rho$ when interchanging one row and column.
The second step of Pan's method, finding a $k\times k$ submatrix of local
$\rho$-maximum volume in $A_\jset$, is the same task as finding the submatrix
$\AI_\B$ in our method. The setting in \cite{pan2000} assumes the dimension of
$A_{11}$ to be given. However, choosing it dynamically by means of a tolerance
$\beta$ as in Algorithm~\ref{alg:maxvolume} can be easily incorporated into the
algorithm for finding the column subset $\jset$. Therefore Pan's method and our
method provide the same functionality.

We consider it an advantage of our method not to use the normal matrix, but
instead to work with the augmented matrix $\AI$. For sparse matrices forming and
factorizing $A_\jset^TA_\jset$ usually leads to more fill-in and can be much
more expensive than factorizing $\AI_\B$.

It will be shown by two examples that normal maximum volume and local maximum
volume are different properties and neither implies the other. First, consider
\begin{equation*}
  A = \bmat{1&&&1\\ &1&&1\\ &&1&1\\ 1&1&1\\ 1&&&1\\ &1&&1\\ &&1&1}
\end{equation*}
and let $A_{11}$ be the leading $3\times3$ block. It can be computed
analytically that the singular values of any three columns of $A$ are
$\ls{\sqrt{5}, \sqrt{2}, \sqrt{2}}$, so that the first three columns have local
maximum volume in $A$. From Lemma~\ref{lem:colchange} it is obvious that
$A_{11}$ has local maximum volume within the first three columns. Hence it has
normal maximum volume in $A$. However, it can be verified from
Lemma~\ref{lem:rowcolchange} that $A_{11}$ does not have maximum volume in the
leading $4\times4$ block and therefore does not have local maximum volume in
$A$.

For the opposite part consider
\begin{equation}
  \label{eq:example2}
  A = \bmat{1&0&0\\ 0&1&0\\ d&-1&-d\\ -1&d&-d}
\end{equation}
with $d=0.99$ and let $A_{11}$ be the leading $2\times2$ block. It can be
verified from Lemma~\ref{lem:rowcolchange} that $A_{11}$ has local maximum
volume in $A$. By computing singular values we obtain the volume of the matrix
composed of columns 1 and 2 to be $2.2272$ and the volume of the matrix composed
of columns 1 and 3 to be $2.4169$.
Hence the first two columns do not have local maximum volume in $A$, and
$A_{11}$ does not have normal maximum volume in $A$.

More insight into the difference between normal and local maximum volume is
obtained from the characterization \cite[Example~2.1]{ben-israel1992} of the
volume of a rectangular matrix. Let $A_\jset\in\RR^{m\times k}$ have rank $k$.
Then
\begin{equation*}
  \vol(A_\jset) = \lt( \sum_{B} \vol(B)^2 \rt)^{1/2},
\end{equation*}
where the sum runs over all nonsingular $k\times k$ submatrices of $A_\jset$.
Hence a subset of $k$ columns has local maximum volume in $A\in\RR^{m\times n}$
if exchanging a column does not increase the ``Euclidean mean'' volume of its
$k\times k$ submatrices. In contrast, let $A_{11}$ have local maximum volume in
$A$ and $A_\jset$ be the column subset that contains $A_{11}$. Then exchanging a
column of $A_\jset$ does not increase $\vol(A_{11})$ or $\vol(B)$ for any $B$
that is neighbour to $A_{11}$ (i.\,e.\ $B$ is obtained by replacing one row of
$A_{11}$ by a row of $A_\jset$ not in $A_{11}$). This property is an immediate
consequence of the definition of local maximum volume.

Let $A_{21}$ denote the lower left $2\times2$ block in \eqref{eq:example2},
which is not neighbour to $A_{11}$. Exchanging columns 2 and 3 in $A$ changes
the volume of $A_{21}$ by a factor
\begin{equation*}
  \vol\lt(\bmat{d&-d\\-1&-d}\rt) / \vol\lt(\bmat{d&-1\\-1&d}\rt) = 99
\end{equation*}
and also increases the Euclidean mean volume of the $2\times2$ submatrix of
$A_\jset$. Therefore the first two columns do not have local maximum volume in
$A$.

\section{Implementation and Results}
\label{sec:results}

We have implemented a simplicial version of Algorithm~\ref{alg:maxvolume} in C
code\footnote{http://www.maths.ed.ac.uk/ERGO/LURank}. By ``simplicial'' we mean
that the implementation does not work on block submatrices and makes no use of
optimized BLAS. It therefore is slower than an optimized singular value
decomposition. Our interest is to examine the number of pivot operations
required and to verify the reliability of the method. Discussing an optimized
implementation is beyond the scope of the paper.

Initially the matrix $W=\bmat{A&I_m}$ is stored. The logical columns are not
explicitly scaled by $\beta$ to avoid values with very different order of
magnitude in the computation. Instead multiplications with $\beta$ and
$\beta^{-1}$ are applied on the fly when logical columns are involved.

In each iteration the algorithm chooses a pivot element in the following order:
\begin{enumerate}
\item If $\abs{W}$ has entries corresponding to block $A_{11}^{-1}$ in
  \eqref{eq:tbl} that are larger than $\rho\beta^{-1}$, then the
  maximum such entry is chosen as pivot.
\item If $\abs{W}$ has entries corresponding to block $A_{11}^{-1}A_{12}$ or
  $-A_{21}A_{11}^{-1}$ in \eqref{eq:tbl} that are larger than $\rho$, then
  the maximum such entry is chosen as pivot.
\item If $\abs{W}$ has entries corresponding to block $A/A_{11}$ in
  \eqref{eq:tbl} that are larger than $\rho\beta$, then the
  maximum such entry is chosen as pivot.
\end{enumerate}
The reason behind the order of choosing a pivot element is to prefer having
logical columns in the basis for numerical stability. If a pivot is found, then
its column is transformed into a unit column by applying row operations to $W$.
If none of the cases (i)--(iii) yields a pivot element, the algorithm
terminates.

The new rank revealing Gaussian elimination algorithm (RRGE) is evaluated on
matrices from the San Jose State University Singular Matrix Database
\cite{foster2018}. We use the 327 matrices (as of January 2018) for which
$\min(m,n)\le1000$. The matrices are transposed if necessary so that $m\le n$.
The parameters used are $\rho=2.0$ and $\beta$ as in \eqref{eq:tol}. For
comparison a singular value decomposition (SVD) of $A$ is computed and the
numerical rank of $A$ is determined as the largest index $s$ such that
\begin{equation}
  \label{eq:svdtol}
  \sigma_s(A) \ge \max(m,n) \eps_{\text{mach}} \sigma_1(A).
\end{equation}
All matrices in the test set are rank deficient by means of \eqref{eq:svdtol}.

For 56 matrices the numerical ranks determined by SVD and RRGE differ. This is
legitimate if there is no large gap between any two consecutive singular values.
To verify that the rank $r$ determined by RRGE is acceptable with respect to the
singular values of $A$, Figure~\ref{fig:comparison-svd} shows the ratios
$\sigma_r(A)/\sigma_s(A)$ and $\sigma_{r+1}(A)/\sigma_{s+1}(A)$ for those
matrices where $r\neq s$. Since the ratios are not too far away from $1.0$, it
can be concluded that
$\sigma_{r+1}(A)=\mathcal{O}(\sigma_{s+1}(A))$ and
$\sigma_r(A)=\Omega(\sigma_s(A))$ and therefore
the rank determined by RRGE is ``correct'' for all matrices in the test set.

\begin{figure}
  \centering
  \includegraphics[width=\textwidth]{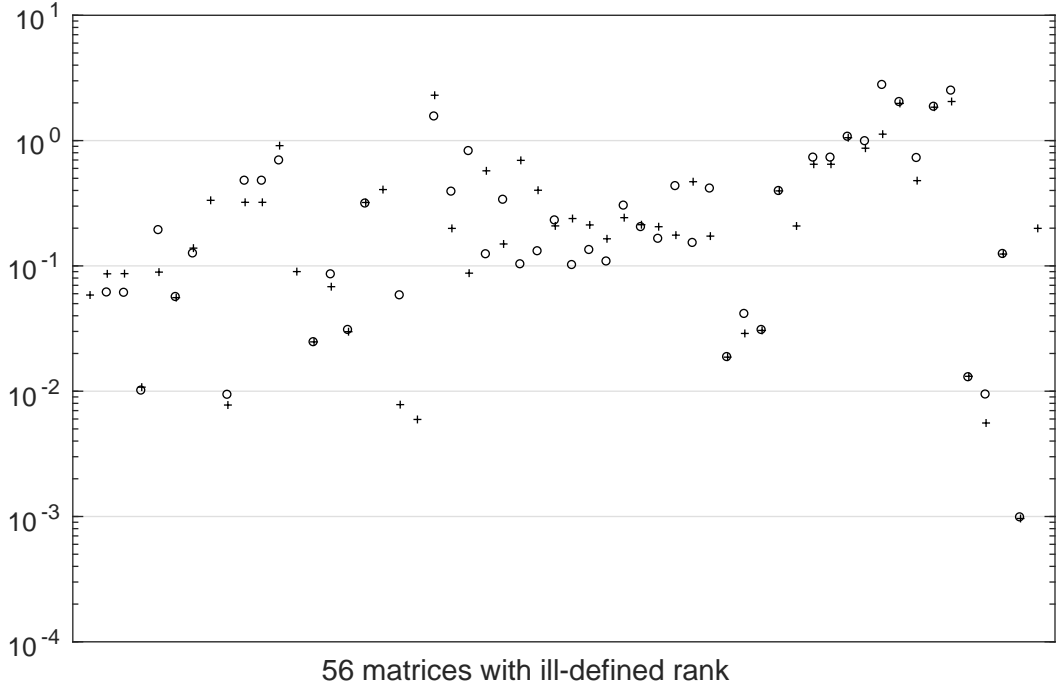}
  \caption{Ratios $\sigma_r(A)/\sigma_s(A)$ (``+'') and
    $\sigma_{r+1}(A)/\sigma_{s+1}(A)$ (``o'') for matrices with $r\neq s$. The
    ``o'' marker is missing when $r=m$ (7 matrices).}
  \label{fig:comparison-svd}
\end{figure}

Table~\ref{tbl:categorized} categorizes the 327 matrices into buckets by means
of $\sigma_r(A_{11})/\sigma_r(A)$ and by the number of pivot operations required
by RRGE. In most cases $\sigma_r(A_{11})$ is much closer to $\sigma_r(A)$ than
Corollary~\ref{cor:lowerbound} guarantees. Because our implementation starts
from the all logical basis, a minimum of $r$ pivots is required. For $\rho=2.0$
the number of pivots is almost always within 5\% of the optimum. The
computational cost for RRGE is roughly twice the cost of an $LU$ factorization
of $\AI$ with complete pivoting. For $\rho=1.1$ the number of pivots
significantly increases on many matrices, but the ratios
$\sigma_r(A_{11})/\sigma_r(A)$ do not improve relevantly.

\begin{table}
  \centering
  \begin{tabular}{ll}
    \begin{tabular}{l|r|r}
      $\sigma_r(A_{11})/\sigma_r(A)$ & $\rho=2.0$ & $\rho=1.1$ \\ \hline
      $(10^{-1}, 10^0]$ & 252 & 255\\
      $(10^{-2}, 10^{-1}]$ & 60 & 69\\
      $(10^{-3}, 10^{-2}]$ & 15 & 3 \\ \multicolumn{3}{c}{}
    \end{tabular}
    &
    \begin{tabular}{l|r|r}
      $\text{pivots}/r$ & $\rho=2.0$ & $\rho=1.1$ \\ \hline
      $[1.00,1.05)$ & 325 & 159 \\
        $[1.05,1.50)$ & 2 & 124 \\
          $[1.5,4.0)$ & 0 & 37 \\
            $[4.0,5.0)$ & 0 & 7
    \end{tabular}
  \end{tabular}
  \caption{Matrices categorized by $\sigma_r(A_{11})/\sigma_r(A)$ and by number
    of pivot operations.}
  \label{tbl:categorized}
\end{table}

\section{Conclusions}
\label{sec:conclusions}

We have presented an algorithm for revealing the numerical rank of $A$ by
Gaussian elimination on the matrix $\bmat{A&\beta I_m}$. The bounds on the
revealed singular values are very similar to those given in \cite{pan2000}, but
our algorithm does not make use of the normal matrix. A prototype implementation
has shown that the number of pivot operations required in practice is only
slightly larger than the rank of $A$. Because the algorithm allows some
flexibility in choosing pivot elements, it can be implemented with blocked
memory access to achieve high floating point performance. An advantage over the
singular value decomposition is to obtain a square nonsingular submatrix and
thereby a maximum set of linearly independent rows and columns. A rank revealing
factorization for sparse matrices based on the results from this paper is a
topic for further research.

\bibliography{rank_reveal}{}
\bibliographystyle{plain}
\end{document}